\documentclass[11pt,reqno]{amsart}

\usepackage{amsmath, amsfonts, amssymb, amscd, enumerate}
\usepackage{color}
\usepackage[all]{xy}

\theoremstyle{plain}
\newtheorem{theo}{Theorem}[section]

\newtheorem{prop}[theo]{Proposition}

\theoremstyle{definition}
\newtheorem{rem}[theo]{Remark}
\newtheorem{example}[theo]{Example}
\newtheorem{definition}[theo]{Definition}

\theoremstyle{plain}
\newtheorem{lemma}[theo]{Lemma}
\newtheorem{theorem}[theo]{Theorem}
\newtheorem{corollary}[theo]{Corollary}
\newtheorem{proposition}[theo]{Proposition}

\theoremstyle{definition}

\newtheorem{remark}[theo]{Remark}

\newcommand{\beq}{\begin{equation}}
\newcommand{\eeq}{\end{equation}}

\newcommand{\C}{\mathbb{C}}

\newcommand{\R}{\mathbb{R}}
\renewcommand{\H}{\mathbb{H}}
\newcommand{\Z}{\mathbb{Z}}

\newcommand\GL{\mathrm{GL}}

\newcommand{\detH}{{\det}_{\mathbb H}}

\renewcommand{\square}{\kern1pt\vbox
{\hrule height 0.6pt\hbox{\vrule width 0.6pt\hskip 3pt
\vbox{\vskip 6pt}\hskip 3pt\vrule width 0.6pt}\hrule height0.6pt}\kern1pt}

\newcommand{\be}{\begin{equation}}
\newcommand{\ee}{\end{equation}}

\def\<#1,#2>{\langle\,#1,\,#2\,\rangle}
\newcommand{\arr}{\begin{array}{rlll}}
\newcommand{\ea}{\end{array}}
\newcommand{\bea}{\begin{eqnarray}}
\newcommand{\eea}{\end{eqnarray}}
\newcommand{\bean}{\begin{eqnarray*}}
\newcommand{\eean}{\end{eqnarray*}}

\def\sideremark#1{\ifvmode\leavevmode\fi\vadjust{
\vbox to0pt{\hbox to 0pt{\hskip\hsize\hskip1em
\vbox{\hsize3cm\tiny\raggedright\pretolerance10000
\noindent #1\hfill}\hss}\vbox to8pt{\vfil}\vss}}}

\newcounter{ssig}
\newcounter{ttig}
\begin{document}
\author{Graziano Gentili}
\address{Dipartimento di Matematica e Informatica ``U. Dini'', Universit\`a di Firenze, 50134 Firenze, Italy}
\email{graziano.gentili@unifi.it}

\author{Anna Gori}
\address{Dipartimento di Matematica, Universit\`a di Milano, Via Saldini 50, 20133 Milano, Italy}
\email{anna.gori@unimi.it}

\author{Giulia Sarfatti}
\address{Dipartimento di Matematica e Informatica ``U. Dini'', Universit\`a di Firenze, 50134 Firenze, Italy} 
\email{giulia.sarfatti@unifi.it}




\title{On compact affine quaternionic curves and surfaces}

\thanks{\footnotesize The authors have been supported by G.N.S.A.G.A. of INdAM - Rome (Italy), Progetto di Ricerca  INdAM ``Teoria delle funzioni ipercomplesse e applicazioni", SIR (Research Project ``Analytic aspects in complex and hypercomplex geometry''), Finanziamento Premiale FOE 2014 (Progetto: ``Splines for accUrate NumeRics: adaptIve models for Simulation Environments-SUNRISE''). The third author has also been supported , by MIUR of the Italian Government (Research Projects: PRIN ``Real and complex manifolds: geometry, topology and harmonic analysis'').}

\maketitle

\begin{abstract}
This paper is devoted to the study of affine quaternionic manifolds and to a possible classification of all compact affine quaternionic curves and surfaces. It is established that on an affine quaternionic manifold there is one and only one affine quaternionic structure. A direct result, based on the celebrated Kodaira Theorem that studies compact complex manifolds in complex dimension $2$, states that the only compact affine quaternionic curves are the quaternionic tori and the primary Hopf surface  $S^3\times S^1$. As for compact affine quaternionic surfaces, we restrict to the complete ones: the study of their fundamental groups, together with the inspection of all nilpotent hypercomplex simply connected $8$-dimensional Lie Groups, identifies a path towards their classification.
\end{abstract}

\section{Introduction}
The definition of slice regularity for functions of one and several quaternionic variables (see, e.g., \cite{libroGSS, Perotti}) has led to a renewed interest for a direct approach to the study of quaternionic manifolds. Quaternionic manifolds, as spaces locally modelled on $\mathbb H^n$ in a slice regular sense, are presented in \cite{GGS} with the name of \emph{quaternionic regular manifolds}, and the closely related class of quaternionic toric manifolds is studied in \cite{GGS1}. In this setting, the class of \emph{affine quaternionic manifolds} - i.e. those manifolds with an atlas whose transition functions are quaternionic affine - 
reveals to be of natural interest, both because of the well established interest for  affine complex manifolds, and for the reason that most of the natural quaternionic manifolds already studied are indeed affine quaternionic manifolds. In his seminal paper, \cite{Sommese}, Sommese introduces the class of \emph{quaternionic  manifolds} as those differentiable $4n$-dimensional  manifolds whose transition functions preserve the standard  right quaternionic structure of $\R^{4n}$. This definition forces the transition maps, and hence the manifolds, to be quaternionic affine. 

The main purpose of this paper is to find a path towards a classification of all compact affine quaternionic curves and surfaces.  

The well celebrated Kodaira Theorem, \cite{Ko}, allows Vitter \cite{Vitter}, Matsushima \cite{Matsu} and Inoue, Kobayashi, Ochiai \cite{IKO} to classify all compact complex manifolds in complex dimension $2$ admitting a complex affine  structure. Since affine quaternionic curves are affine complex surfaces, we can use this classification to identify all $1$-dimensional compact affine quaternionic manifolds.

In quaternionic dimension 2, the lack of a classification of affine compact complex manifolds of dimension 4 advises us to change point of view in order to classify compact affine quaternionic surfaces. To identify the subclass of compact, geodesically complete (or complete for short), affine quaternionic surfaces we adopt in fact the approach used in [7], based on the study of their fundamental groups, and prove the following result. 

\begin{theorem} \label{intro} If the subgroup $\Gamma\subseteq {\it Aff}(2,\H)$
acts freely and properly discontinuously  on $\H^2$, and  $\H^2/\Gamma$ is compact,  then $\Gamma$ contains a unipotent (see Definition \ref{unipotent}) normal subgroup $\Gamma_0$ of finite index such that $\Gamma/\Gamma_0$ is isomorphic to  a finite subgroup of $\mathbb{S}^3$. 
\end{theorem}
\noindent A \emph{nilmanifold} is defined as a compact coset space of the form $N/K$, where $N$ is a connected, simply connected nilpotent Lie group and $K$ is a discrete subgroup ($K$ is the fundamental group of the nilmanifold). A theorem of Malcev, \cite{Mal}, states that a given abstract group is the fundamental group of a suitable nilmanifold if and only if it is a finitely generated nilpotent group with no elements of finite order. Using this result we obtain further properties of the subgroup $\Gamma_0$ appearing in Theorem \ref{intro}:

\begin{corollary} 
In the hypotheses of Theorem \ref{intro}, the unipotent subgroup $\Gamma_0$ of $\Gamma$ turns out to be isomorphic to a discrete subgroup $i(\Gamma_0)$ of a suitable nilpotent, hypercomplex, connected, simply connected, $8$-dimensional Lie group $N$ such that $N/i(\Gamma_0)$ is compact. Moreover $N/i(\Gamma_0)$ is affine quaternionic, and the same holds for $N$.
\end{corollary}
\noindent The corollary asserts in particular that any geodesically complete, compact, affine quaternionic surface is finitely covered by a suitable nilmanifold whose fundamental group is isomorphic to $\Gamma_0$.
This fact - together with the classification of all nilpotent hypercomplex simply connected $8$-dimensional Lie Groups given by Dotti and Fino, \cite{DF} - indicates a path towards the classification of compact, complete, affine quaternionic surfaces. As an application, in the last Section we are able to classify all affine quaternionic surfaces arising in the case of a real Heisenberg group.


\section{affine quaternionic manifolds} 
\noindent In this setting  the Dieudonn\'e determinant $\det_{\mathbb{H}}$ plays a similar role as the usual one. To each quaternionic matrix we can associate a complex matrix via the algebra homomorphism
	\[\psi: \mathcal{M}(n,\H)\to\mathcal{M}(2n,\C)\]
	defined by
	\[\psi(A+Bj)=\begin{pmatrix}
	A&-\overline{B}\\
	B&\overline{A}
	\end{pmatrix},\]
and it turns out that  $(\det_{\mathbb{H}}(M))^2=\det(\psi(M))\geq 0$,
where the right hand term is the usual determinant of $\psi(M)$, see \cite{A}. 
Hence, the group of quaternionic $n\times n$ invertible matrices $GL(n, \H)$ can be introduced in the usual fashion via the Dieudonn\'e determinant. 

For $Q=\,^t(q_1,q_2,\ldots,q_n)\in \H^n$,  we can define the group of all quaternionic affine transformations 
$$
{\it Aff}(n,\H)=\{Q \mapsto AQ+B : A\in GL(n,\H), B=\,^t(b_1, b_2, \ldots , b_n)\in \H^n \}
$$ 
which is included in the class of (right) slice regular functions, \cite{Perotti}. In complete analogy with what Kobayashi does in the complex case, we give the following: 

\begin{definition}
A differentiable manifold $M$ of $4n$ real dimensions has  a \emph{quaternionic affine structure} if it admits a differentiable atlas whose transition functions are restrictions of quaternionic affine functions of ${\it Aff}(n, \H)$.     
\end{definition}

\noindent In particular, differentiable manifolds endowed with a  quaternionic affine structure are quaternionic regular \cite{GGS}, \cite{GGS1}. The following result allows to describe the entire class of quaternionic affine manifolds. 

\begin{lemma}Let $M$ be a differentiable manifold of dimension $4n$. Then $M$ admits an affine  quaternionic structure if, and only if, there is an immersion $\psi : \widetilde M \to \H^n$ of the universal covering $\widetilde M$ of $M$ such that, for every covering transformation $\gamma$ we have $\psi\circ\gamma = X_\gamma \circ \psi$ for some affine transformation $X_\gamma$ of $\H^n$.
\end{lemma}
\noindent The holomorphic analogue of the previous lemma, can be found, e.g., in \cite{KH83}. We recall that the quaternionic manifolds studied by Sommese, \cite{Sommese}, all admit a quaternionic affine structure. Many significant examples can be found in his paper.

A large class of affine quaternionic manifolds can be constructed by means of a subgroup $\Gamma \subset {\it Aff}(n, \H)$ which acts freely and properly discontinuously on $\H^n$. Indeed the quotient space
\[
M=\H^n\slash \Gamma
\]
admits an atlas whose transition functions are the slice regular functions belonging to $\Gamma \subset {\it Aff}(n,\H)$, and hence has  a quaternionic affine structure.  The latter class of manifolds consists of all geodesically complete quaternionic affine manifolds, see, e.g., \cite{Goldman}.

The relation between  affine complex structures and  flat connections in the complex setting has been  deeply investigated  during the past years.  In particular in the complex setting a theorem of Matsushima  \cite{Matsu} states that  there is a one-to-one correspondence between  affine  structures and  affine holomorphic connections  which are torsion-free and flat. 
Vitter \cite{Vitter}, Inoue, Kobayashi e Ochiai \cite{IKO} gave a classification of all manifolds admitting such connections in complex dimension $1$ and $2$.
\\
A similar  correspondence holds also in the quaternionic setting. But the quaternionic structures are much more rigid. A manifold is said to admit a $GL(n,\H)$-structure if it can be endowed with two anticommunting almost complex structures, see \cite[page 48]{Salamon}. On such manifolds,  also called   {\em almost quaternionic} by Sommese in \cite{Sommese}, it is possible to define  a connection, the Obata connection, which is torsion free, and it is the only one with this property. Moreover the Obata connection turns out to be flat if and only if the  $GL(n,\H)$-structure is integrable (if and only if $M$ is quaternionic in the sense of Sommese). 

For an almost quaternionic manifold, having an integrable $GL(n,\H)$-structure 
%
is equivalent to the nullity of three tensors which, in the quaternionic setting, play the role of the Nijenhuis tensor,  \cite{Obata}.
Summarizing

\begin{proposition} A manifold $M$ has an integrable $GL(n,\H)$-structure if and only if  it is affine quaternionic and equivalently if and only if the torsion-free Obata connection of $M$ is flat.
\end{proposition}
\noindent Moreover 

\begin{remark} \label{hypercomplex} An affine quaternionic manifold is {\em hypercomplex}, since the integrability of the $GL(n,\H)$-structure implies that it can be endowed with two anti-commuting complex structures. 
\end{remark}
\noindent Thanks to the one-to-one correspondence between affine structures and flat torsion free holomorphic connections in the complex setting, and the uniqueness of the Obata connection on an affine quaternionic manifold,  we obtain that

\begin{corollary} On an affine quaternionic manifold there is one and only one affine quaternionic structure.
\end{corollary}
\noindent In the complex setting the situation is quite different. Indeed
a fixed affine compact complex manifold may have a number of distinct affine structures which all induce the given complex structure, that is, it may have affine structures which are complex analytically but not affinely equivalent. As an example of this phenomenon, consider a complex $1$-dimensional torus $T = {\C}/{\Lambda}$. The usual affine coordinate on
$T$ is the coordinate $z$ of the universal cover of $T$,  defined locally on $T$. But there are other distinct affine structures on $T$, in fact, for all $a\in \C$, $a\neq 0$, there is an affine structure on $T$ whose coordinate is
$$\frac{1}{a}(e^{az}-1)=\sum_k^{\infty} \frac{a^{k-1}z^k}{k!}$$
where $z$ is the usual affine coordinate mentioned above.

\section{Towards a classification of affine quaternionic manifolds in low dimension}

In order to classify all the affine quaternionic manifolds in low dimensions, one can  try to argue as in the complex case. If $M$ is an affine quaternionic manifold of quaternionic dimension $1$, i.e., an \emph{affine quaternionic curve},  it is also a complex affine surface.  
In complex dimension $2$ the Kodaira Theorem gives a list of compact complex manifolds on which one can study the affine structures or, equivalently, the affine holomorphic flat connections.  Indeed Vitter, \cite{Vitter}, goes through the seven classes identified by Kodaira, and lists explicitly all possible affine complex structures on the following classes of affine complex surfaces: tori, affine Hopf surfaces, quotients of Abelian varieties (by cyclic groups of order $2,3,4$ or $6$), fiber bundles of $1$-dimensional tori over a $1$-dimensional torus, and quotients of these bundles.

Among the complex linear transformations of $\C^2$ those that are also quaternionic linear transformations  of $\H$ are represented, as we have seen, by non zero matrices of the form
\[
\begin{pmatrix}
	a&-\overline{b}\\
	b&\overline{a}
	\end{pmatrix}
\]
with $a, b \in \C$. Examining Vitter's classification, it is not difficult to see that there are no other quaternionic curves, with the exception of the quaternionic tori that have  already been studied in \cite{BG} and the primary Hopf surface $S^3\times S^1$.

In quaternionic dimension $2$, some examples of affine quaternionic manifolds are given by slice affine quaternionic Hopf surfaces, \cite{AB},  and by some tori that one can construct by adapting the strategy  in  \cite{Vitter} to complex dimension $4$. 


Since we cannot refer to a classification of affine, compact, complex manifolds of dimension $4$, we restrict to the class of complete, affine quaternionic manifolds in dimension $2$, i.e, those of the form $\H^2\slash \Gamma$ with $\Gamma\subset {\it Aff}(2, \H)$ acting freely and properly discontinuously on $\H^2$. In this setting we adopt the approach used in \cite{FS}, based on the study of the fundamental groups of affine compact complex surfaces.
Our aim, indeed, is to find necessary conditions on a discrete subgroup  $\Gamma$ of the group of quaternionic  affine transformations ${\it Aff}(2,\mathbb{H})$ so that its action on $\mathbb{H}^2$ is free and properly discontinuous. In what follows we identify ${\it Aff}(2,\mathbb{H})$  with the group of invertible $3 \times 3$ matrices with quaternionic entries of the form
$$\begin{pmatrix}
a&b&r\\
c&d&s\\
0&0&1\\
\end{pmatrix}$$

\noindent In this section we study  properties of a subgroup $\Gamma$ of the group ${\it Aff}(2,\mathbb{H})$, acting {\em freely} on $\mathbb{H}^2$.
The action of $A\in\Gamma$ on $\mathbb{H}^2$ on the left maps $(x,y)$ in $(x',y')$ where
$$\begin{cases}x'=ax+by+r\\
y'=cx+dy+s
\end{cases}$$
With the usual notation, we denote with $h(A)$ the matrix $\begin{pmatrix}
a&b\\
c&d\\
\end{pmatrix}$ in $GL(2,\mathbb{H})$, called the {\em holonomy part} of $A$.
\\

We first recall a few definitions  and well known facts about eigenvalues and eigenvectors in the quaternionic setting. We refer to \cite{deLeo},\cite{leib} for an exhaustive treatment of quaternionic linear algebra.
In the study of spectral theory in  the quaternionic setting one has to define what is an eigenvalue for a matrix $A$, indeed, once chosen the left action of the matrix, one can state the ``right eigenvalue problem'' and the ``left eigenvalue problem'' according to the position of the eigenvalue. We will focus on the right eigenvalue problem.
\begin{definition} Let $A$ be a $n\times n$-quaternionic matrix. Then $\lambda\in\mathbb{H}$ is a {\em right eigenvalue} for $A$ if and only if  there exists a nonzero $v\in \mathbb{H}^n$ such that
$$Av=v\lambda.$$ In this case $v$ is called an {\em eigenvector} of $A$.
\end{definition}
\rem If $\lambda \in \H$ is an eigenvalue of a quaternionic matrix $A$, then  all the elements in the $2$-sphere $S_\lambda=\{u^{-1}\lambda u : 0\neq u \in \H\}$ of all conjugates of $\lambda$ turn out to be eigenvalues of $A$: if $Av=v\lambda$, then $A(vu)=(vu)u^{-1}\lambda u$ for any invertible $u\in \mathbb{H}$.
\rem \label{multeigen}If $v$ is an eigenvector of a quaternionic matrix $A$, with eigenvalue $\lambda$, then  $v\mu$ with $\mu\in \mathbb{H}, \mu\neq 0$, is  an eigenvector with respect to the eigenvalue $\mu^{-1}\lambda \mu$ in the sphere $S_\lambda$.
\\

 	\begin{prop}
Let $M\in \mathcal{M}(n,\H)$ be a quaternionic matrix. Then $\lambda\in\H$ is a right eigenvalue of $M$ if and only if there exists a complex $\tilde\lambda\in S_\lambda$ such that
\[\detH(M-\tilde\lambda I_n)=0.\]
	\end{prop}
\begin{proof}
The quaternion $\lambda$ is a right eigenvalue of $M$ if and only if every element in $S_\lambda$ is. Let $\tilde \lambda \in S_\lambda$ be a complex eigenvalue of $M$; then $\tilde \lambda$ is an eigenvalue of $\psi(M)$, and hence $$0=\det (\psi(M)-\tilde \lambda I_{2n})=\det (\psi(M-\tilde \lambda I_{n}))=(\detH (M-\tilde \lambda I_n))^2.$$

\end{proof}	

\noindent We point out that right eigenvalues are shared by similar matrices: if $Av=v\lambda$, then $M^{-1}AM(M^{-1}v)=(M^{-1}v)\lambda$ for any invertible quaternionic matrix $M$. The same is not true when considering left eigenvalues. In addition, a quaternionic matrix is diagonalisable if and only if its complex representation is diagonalisable. 

\begin{lemma}
Let the subgroup $\Gamma \subseteq {\it Aff}(2,\mathbb{H})$ act freely on $\mathbb{H}^2$. Then each element of $h(\Gamma)$ has $1$ as an eigenvalue.\end{lemma}
\begin{proof} Let $A\in\Gamma$. The point $(x,y)\in \mathbb{H}^2$ is fixed by 
	\[A=\begin{pmatrix}
	a&b&r\\
	c&d&s\\
	0&0&1\\
	\end{pmatrix}\] if and only if
$$\begin{cases}(a-1)x+by=-r\\
cx+(d-1)y=-s
\end{cases}$$ 
If $1$ is not an eigenvalue of $h(A)$, then $(A-I) \in GL(2, \H)$ and hence the linear system has a solution; thus the action is not  free.
%
\end{proof}

\noindent Let us now define two groups  of quaternionic matrices,
\[
G_1=\left\{\begin{pmatrix}
a&b&r\\
0&1&s\\
0&0&1\\
\end{pmatrix}: a,b,r,s, \in \mathbb H,  a\neq 0\right\}
\]
 and
\[
G_2=\left\{\begin{pmatrix}
1&b&r\\
0&d&s\\
0&0&1\\
\end{pmatrix} :
 b,r,s,d \in \mathbb H, d\neq 0 \right\},
\] 
which play a key role in the study of subgroups of ${\it Aff}(2,\mathbb{H})$ acting freely on $\mathbb{H}^2$.
 \begin{proposition}Let the subgroup $\Gamma\subseteq {\it Aff}(2,\mathbb{H})$ act freely on $\mathbb{H}^2$. Then $\Gamma$ is conjugate in ${\it Aff}(2,\mathbb{H})$ to a subgroup of $G_1$ or $G_2.$
\end{proposition} 
\begin{proof}
Suppose first that $\Gamma$ contains an element $A$ such that $h(A)$ has an eigenvalue $\lambda\neq 1$. Then, we can diagonalise $h(A)$  via a matrix $P$ in $\GL(2,\mathbb{H})$. Suppose $$B\in\begin{pmatrix}
P&0\\
0&1&\\
\end{pmatrix}\Gamma\begin{pmatrix}
P&0\\
0&1&\\
\end{pmatrix}^{-1}$$
Write $h(B)=\begin{pmatrix}
a&b\\
c&d\\
\end{pmatrix},$ then $Ph(A)P^{-1}h(B)=\begin{pmatrix}
\lambda & 0\\
0&1\\
\end{pmatrix}h(B)=
\begin{pmatrix}
\lambda a&\lambda b\\
c&d\\
\end{pmatrix}.$
By the previous lemma, both $h(B)$ and $Ph(A)P^{-1}h(B)$ have $1$ as an eigenvalue, so there exist $(x,y)$ and $(z,w) \in\mathbb H^2$ such that
$$\begin{cases}ax+ b y=x\\
cx+dy=y
\end{cases}
 \quad \text{and} \quad \begin{cases}\lambda az+\lambda b w=z\\
cz+dw=w
\end{cases}.$$ 
Suppose first that $y\neq 0$ and $w\neq 0$. Hence, up to a rescaling of the eigenvector (note that, in general, thanks to Remark \ref{multeigen}, the corresponding eigenvalue changes, remaining in the same sphere; in the present case the real eigenvalue does not change), we can suppose that $y=w$. In this case, subtracting the second equations of the systems, we get $c(x-z)=0$ which implies either $c=0$ or $x=z$.
\begin{itemize}
\item  If $c=0$ then  $d=1$ (since $y\neq 0$);
\item  If $x=z$ we get $\lambda=1$ (a contradiction) or $x=z=0$.  If $x=z=0$ then $b=0$ (and hence again $d=1$ since $y\neq 0$).
\end{itemize}

\noindent Suppose now that $x\neq0$ and $z\neq 0$; then again we can assume that $x=z$, and with straightforward computations we get $d(y-w)=y-w$; thus $d=1$ or $y=w$. 
\begin{itemize}
\item If $d=1$ then $c=0$ (since $x\neq 0$);
\item if $y=w\neq 0$ we get $\lambda=1$ (a contradiction) or $y=w=0$. If $y=w=0$ then $c=0$ and $a=1$.
\end{itemize}
If  now $y=0$ (and necessarily $x\neq 0$) and $z=0$ (and necessarily $w\neq 0$), we get $c=0$ and $a=1$. 
If instead  $x=0$ (and necessarily $y \neq 0$) and $w=0$ (and necessarily $z\neq 0$, we get $b=0$ and $d=1$.  

\noindent So the possibilities for $h(B)$ are: \noindent if $b=0$,
$$\begin{pmatrix}
a& 0\\
c&1\\
\end{pmatrix};
$$
or, if $c=0$
$$\begin{pmatrix}
a& b\\
0&1\\
\end{pmatrix};
$$
Note that we cannot have both kinds of $h(B)$ occurring, for if both $\begin{pmatrix}
a& b\\
0&1\\
\end{pmatrix}, \begin{pmatrix}
a'& 0\\
c'&1\\
\end{pmatrix}$ where in $Ph(\Gamma)P^{-1}$ with $b\neq 0$ and $c'\neq 0$ also their product   
$\begin{pmatrix}
aa'+bc'& b\\
c'&1\\
\end{pmatrix}$ would belong to it, but it is easy to prove that this matrix  does not have $1$ as eigenvalue.
Hence we have that 
\[
\begin{pmatrix}
P&0\\
0&1&\\
\end{pmatrix}\Gamma\begin{pmatrix}
P&0\\
0&1&\\
\end{pmatrix}^{-1}
\]is contained in $G_1$ or in the group of all quaternionic matrices of the form $ \begin{pmatrix}
a& 0&r\\
c&1&s\\
0&0&1\\
\end{pmatrix};$ the latter is conjugate  to $G_2$ via an element of type 
$ \begin{pmatrix}
0& 1&0\\
1&0&0\\
0&0&1\\
\end{pmatrix}$
and we are done.\\

Now suppose every element of $h(\Gamma)$ has both eigenvalues $1$. If $h(\Gamma)$ is the identity, we are done, otherwise some conjugate of $\Gamma$ contains an element of the form $ L=\begin{pmatrix}
1& 1&u\\
0&1&v\\
0&0&1\\
\end{pmatrix}.$
Let $ C=\begin{pmatrix}
a& b&r\\
c&d&s\\
0&0&1\\
\end{pmatrix}$
be an arbitrary element of this conjugate of $\Gamma$.
Then $h(C)=\begin{pmatrix}
a&b\\
c&d\\
\end{pmatrix}$ and $h(LC)=\begin{pmatrix}
a+c&b+d\\
c&d\\
\end{pmatrix}$ have $1$ as eigenvalue with molteplicity $2$. A direct computation implies again that $c=0$ and $a=d=1$.
\end{proof}
In the complex case, $G_1$ and $G_2$ turn out to be solvable; we point out that this is not the case in the quaternionic setting due to the non commutativity of $\mathbb{H}$.
\begin{lemma} \label{abel1}If the subrgroup $\Gamma\subseteq {\it Aff}(2,\mathbb{H})$ acts freely on $\H^2$ and $h(\Gamma)$ is abelian then $\Gamma$ is conjugate in ${\it Aff}(2,\H)$ to a subgroup of the group of all matrices of the form 
$ \begin{pmatrix}
1& 0&r\\
0&d&s\\
0&0&1\\
\end{pmatrix}$ with $d\neq 0$ or  to a subgroup of the group of all matrices of the form $\begin{pmatrix}
1& b&r\\
0&1&s\\
0&0&1\\
\end{pmatrix}$
\end{lemma}
\begin{proof}
The proof  given for the complex case in [Fillmore, Lemma 2.4] can be easily adapted to matrices with quaternionic entries.
\end{proof}
\begin{lemma}\label{2.8} If $ \begin{pmatrix}
a& b&r\\
0&1&0\\
0&0&1\\
\end{pmatrix}$
has no fixed points in $\H^2$ then $a=1$ and $b=0$. 
\end{lemma}
\begin{proof} If $b\neq 0$  then $(0,-b^{-1}r,1)$ is a  fixed point. Now, suppose $b=0$; if $a\neq 1$ then $(-(a-1)^{-1}r,y,1)$ is a fixed point. Hence the assertion follows.
\end{proof}
\begin{lemma} \label{abel2} If the subgroup $\Gamma \subseteq G_1$ acts freely on $\H^2$ then $h(\Gamma)$ is abelian and there exists a complex plane containing all $a$ such that \[\begin{pmatrix}
	a& b\\
	0&1
	\end{pmatrix}\in h(\Gamma).\]
\end{lemma}

\begin{proof} Let $$A= \begin{pmatrix}
a& b&r\\
0&1&s\\
0&0&1\\
\end{pmatrix}\;\;B= \begin{pmatrix}
a'& b'&r'\\
0&1&s'\\
0&0&1\\
\end{pmatrix}$$ be elements of $\Gamma$. By direct computation, it is easy to verify that their inverse elements  are the quaternionic matrices
$$A^{-1}= \begin{pmatrix}
a^{-1}& -a^{-1}b&-a^{-1}r+a^{-1}bs\\
0&1&-s\\
0&0&1\\
\end{pmatrix}$$
$$B^{-1}=
 \begin{pmatrix}
{a'}^{-1}& -{a'}^{-1}b'&-{a'}^{-1}r'+{a'}^{-1}b's'\\
0&1&-s'\\
0&0&1\\
\end{pmatrix}.$$
Moreover, for some $c\in \mathbb H$, $$ABA^{-1}B^{-1}= 
\begin{pmatrix}
aa'{a}^{-1}{a'}^{-1}&aa'(-a^{-1} {a'}^{-1}b'-a^{-1}b)+ab'+b&c\\
0&1&0\\
0&0&1\\
\end{pmatrix}.$$
From the fact that $ABA^{-1}B^{-1}$ acts without fixed points, applying Lemma \ref{2.8}, we get that $$(ABA^{-1}B^{-1})_{12}=0\quad{\rm and}\quad aa'{a}^{-1}{a'}^{-1}=1$$ which immediately imply that $a$ and $a'$ belong to the same complex plane, and that $h(\Gamma)$ is abelian.
 \end{proof}
 \noindent Now, combining Lemmas \ref{abel1} and \ref{abel2}, we get
 \begin{corollary} If the subgroup $\Gamma\subseteq {\it Aff}(2,\H)$ acts freely on $\H^2$, then $\Gamma$ is conjugate in ${\it Aff}(2,\H))$ to a subgroup of $G_2.$
 \end{corollary}
 
 \begin{lemma}\label{abel3} If the subgroup $\Gamma\subseteq {\it Aff}(2,\H)$
is abelian and acts freely on $\H^2$ then it is conjugate in ${\it Aff}(2,\H)$ to a subgroup of the group of all matrices of the form
$ \begin{pmatrix}
1& 0&r\\
0&d&0\\
0&0&1\\
\end{pmatrix}$ with $d\neq 0$ or  to a subgroup of the group of all matrices of the form $\begin{pmatrix}
1& b&r\\
0&1&s\\
0&0&1\\
\end{pmatrix}.$
 \end{lemma}
\begin{proof}Since $h(\Gamma)$ is abelian we can use Lemma \ref{abel1} and conjugate $\Gamma$ into the group of all
$\begin{pmatrix}
1& b&r\\
0&1&s\\
0&0&1\\
\end{pmatrix}$
and we are done, or into the group of all $ \begin{pmatrix}
1& 0&r\\
0&d&s\\
0&0&1\\
\end{pmatrix}.$ In the latter case: if all entries $d=1$ we are done. Otherwise suppose $d\neq1$ for some element $A$ in $\Gamma$.  After a conjugation with  
\[C=\begin{pmatrix}
1& 0&0\\
0&1&(1-d)^{-1}s\\
0&0&1\\
\end{pmatrix},\]
the matrix $A$ is taken to 
$\begin{pmatrix}
1&0&r\\
0&d&0\\
0&0&1\\
\end{pmatrix}$
and all the other elements to 
$A'=\begin{pmatrix}
1& 0&r'\\
0&d'&s'\\
0&0&1\\
\end{pmatrix}.$ 
Now, since $\Gamma$ is abelian, we have that $CAC^{-1}A'=A'CAC^{-1}$ which implies  that $d'd=dd'$ (so $d$ and $d'$ belong to the same complex plane), and $ds'=s'$. Thus if there exists $s'\neq 0$ we get $d=1$, a contradiction.
\end{proof}

\noindent In what follows in addiction to the hypothesis that the action of $\Gamma$ on $\H^2$ is {\em free}, we will assume that $\Gamma$ acts {\em properly discontinuously} 
and that $\H^2/\Gamma$ is compact. This has important consequences that we collect here. \\
We recall the First Theorem of Bieberbach, see, e.g., \cite[Theorem 1]{Au}.
\begin{theorem}\label{auslander}
Let $G$ be a subgroup of ${\it Aff}(n,\C)$, acting freely and properly discontinuously on $\C^n$ and such that $\C^n/G$ is compact. Then the subgroup $\widetilde G \subseteq G$ of pure translations is generated by $n$ linearly independent translations and $G/\widetilde G\simeq h(G)$ is a finite group. 
\end{theorem}
As a direct consequence we get
\begin{lemma}\label{base} If the subgroup $\Gamma\subseteq {\it Aff}(2,\H)$ acts freely, properly discontinuously on $\H^2$ and  $\H^2/\Gamma$ is compact, then the set of translational parts $(r,s)$ of elements of the form
$\begin{pmatrix}
a& b&r\\
0&1&s\\
0&0&1\\
\end{pmatrix}$
of $\Gamma$ contains a basis for $\H^2$ as a real vector space.
\end{lemma} 

\begin{proof} 
The proof easily follows taking into account that ${\it Aff}(2,\H)$ can be identified as a subgroup of ${\it Aff}(4,\C)$.   	
\end{proof}
\noindent We can now solve completely the abelian case: indeed, applying the previous Lemma, we understand that the first possibility of Lemma \ref{abel3} does not occur; thus we have
\begin{corollary} If the subgroup $\Gamma\subseteq {\it Aff}(2,\H)$
is abelian and acts freely and properly discontinuously  on $\H^2$, and  $\H^2/\Gamma$ is compact,  then it is conjugate in ${\it Aff}(2,\H)$ to a subgroup of the group of all matrices of the form
$ \begin{pmatrix}
1& b&r\\
0&1&s\\
0&0&1\\
\end{pmatrix}.$
\end{corollary}
\begin{lemma} \label{nonabel} If the subgroup $\Gamma\subseteq {\it Aff}(2,\H)$
acts properly discontinuously  on $\H^2$, and  contains elements $$A=\begin{pmatrix}
1& b&r\\
0&d&s\\
0&0&1\\
\end{pmatrix} \rm{and}\;\; B=\begin{pmatrix}
1& f&u\\
0&h&v\\
0&0&1\\
\end{pmatrix}$$ such that $AB\neq BA$, then $d$ is a root of unity.
\end{lemma}
\noindent The proof is similar to the one in the complex case, but the computations are much more complicated, due to the non commutativity of quaternions. 
\begin{proof}
By direct computation we obtain that for any $n\in\mathbb{N}$ 
\[A^n=\begin{pmatrix}
1& b(d-1)^{-1}(d^n-1)&b(d-1)^{-2}(d^n-1)s+nr-bn(d-1)^{-1}s\\
0&d^n&(d-1)^{-1}(d^n-1)s\\
0&0&1\\
\end{pmatrix}\]
Hence
\[C_n=A^{-n}BA^{n}B^{-1}=\begin{pmatrix}
1& f_n&u_n\\
0&h_n&v_n\\
0&0&1\\
\end{pmatrix} \]
where
\begin{equation*}
\begin{aligned}
h_n&=d^{-n}hd^nh^{-1}\\
f_n&=f(d^n-1)h^{-1}+b(d-1)^{-1}(d^{n}-1)(1-d^{-n}hd^n)h^{-1}\\
u_n&=fh^{-1}v-fd^nh^{-1}v-b(d-1)^{-1}(d^n-1)h^{-1}v\\
&+b(d-1)^{-1}(d^n-1)d^{-n}hd^nh^{-1}v+f(d-1)^{-1}(d^n-1)s\\
&-b(d-1)^{-1}(d^n-1)d^{-n}h(d-1)^{-1}(d^{n}-1)s\\
&-b(d-1)^{-1}(d^n-1)d^{-n}v+b(d-1)^{-2}(d^n-1)^2d^{-n}s\\
v_n&=-d^{-n}hd^{n}h^{-1}v+d^{-n}h(d-1)^{-1}(d^n-1)s+d^{-n}v-d^{-n}(d-1)^{-1}(d^n-1)s
\end{aligned}
\end{equation*}
Suppose that $d$ is not a root of unity. Let us show that, in this case, $C_n\neq C_m$ for $n\neq m$. 

Assume first that $hd\neq dh$. Then $C_n=C_m$ if and only of $h_n=h_m$, that is if and only if $n=m$.

If instead $h$ and $d$ commute, then the entries of $C_n$ become
\begin{equation*}
\begin{aligned}
h_n&=1\\
f_n&=[fh^{-1}+b(d-1)^{-1}(h^{-1}-1)](d^n-1)\\
u_n&=-f(d^{n}-1)[h^{-1}v-(d-1)^{-1}s]-b(d-1)^{-1}(d^n-1)(h^{-1}v-v)\\
&+b(d-1)^{-2}(2-d^{-n}-d^n)(s-hs)-b(d-1)^{-1}(1-d^{-n})v\\
v_n&=(1-d^{-n})[-v+(d-1)^{-1}(hs-s)]
\end{aligned}
\end{equation*}
Suppose that $C_n=C_m$ for $n\neq m$. Then $f_n=f_m$ and $v_n=v_m$ that is
\begin{equation*}
\left\{\begin{array}{l}
fh^{-1}+b(d-1)^{-1}(h^{-1}-1)=0\\
-v+(d-1)^{-1}(hs-s)=0
\end{array}
\right.
\end{equation*}
which is equivalent to
\begin{equation}\label{sistema}
\left\{\begin{array}{l}
f(d-1)+b(1-h)=0\\
-(d-1)v+(1-h)s=0
\end{array}
\right. 
\end{equation}
The first equation in system \eqref{sistema} is equivalent to the fact that the holonomies $h(A)$ and $h(B)$ commute. Recalling that $d\neq 1$, we have that $h(A)$ can be diagonalised, thus also $h(B)$ can be diagonalised via the same matrix (direct computation: indeed suppose that $h(B)$ cannot be diagonalised, since they commute they can be simultaneously triangulated, and they still commute, thus $f=fd$ which implies $f=0$ since $d\neq 1$). Hence, up to conjugation, 
\[A=\begin{pmatrix}
1& 0&r\\
0&d&s\\
0&0&1\\
\end{pmatrix} \quad \text{and} \quad B=\begin{pmatrix}
1& 0&u\\
0&h&v\\
0&0&1\\
\end{pmatrix}\]
and if the second equation in System $\eqref{sistema}$ is satisfied, $A$ and $B$ commute, in contradiction with our hypothesis. Therefore the matrices $C_n$ are all distinct.


Let us show that, in both cases, if $d$ is not a root of unity, the action of $\Gamma$ is not properly discontinuous. \\
Suppose first that $|d|=1$. 
Consider the sequence of points $(x_n,y_n)\in \H^2$ obtained applying the matrices $C_n$ to the point $(0,1)$,
\begin{equation*}\label{orbita}
\begin{array}{l}
x_n=f_n+u_n\\
y_n=d_n+v_n
\end{array}
\end{equation*}
Up to subsequences, $d^n$ goes to $1$ as $n$ goes to infinity, thus $h_n$ goes to $1$ and $f_n,u_n$ and $v_n$ go to $0$. Hence the orbit of $(0,1)$ has $(0,1)$ as an accumulation point.\\
Suppose now that $|d|>1$, and consider the orbit of $(0,v-h(d-1)^{-1}s)$. In this case 
\begin{equation*}\label{orbita}
\begin{array}{l}
x_n=f_n(v-h(d-1)^{-1}s)+u_n\\
y_n=d_n(v-h(d-1)^{-1}s)+v_n
\end{array}
\end{equation*}
Up to subsequences, $h_n$ tends to $a\in\H$, $|a|=1$ as $n$ tends to infinity  (if $d$ and $h$ commute $h_n=1$ for any $n$). Hence, with long but straightforward computations, we get 
\begin{equation*}\label{orbita2}
 \begin{array}{l}
x_n=b(d-1)^{-1}[(d-1)^{-1}(d^{-n}-1)s+(1-d^{-n})h(d-1)^{-1}s-(1-d^{-n})v]\\
y_n=-d^{-n}hd^n(d-1)^{-1}s + d^{-n}h(d-1)^{-1}(d^{n}-1)s +d^{-n}v-(d-1)^{-1}(1-d^{-n})s
\end{array}
\end{equation*}
Taking into account the fact that $d^{-n}hd^{n}$ is bounded, since in modulus equals $|h|$, we obtain that $(x_n,y_n)$ has an accumulation point.\\
The case where $|d|<1$ can be treated analogously, considering the orbit through the same point via the matrices $\widetilde {C_{n}}=A^nBA^{-n}B^{-1}$.


\end{proof}	
In order to state and prove the next result, we define and list all (up to conjugation) finite subgroups of  unitary quaternions; to do this we use the notations and approach of the book \cite{conway} by Conway and Smith.
Let  $I,J\in  \H$ be purely imaginary unit quaternions, with $I\perp J$, and let $\{1, I, J, IJ=K\}$ be a basis for $\mathbb{H}$ having the usual multiplication rules. For 

\begin{align*}
 \sigma=\frac{\sqrt 5-1}{2},\ \ \  \tau=\frac{\sqrt 5+1}{2}
\end{align*}
we consider the unitary quaternions

\begin{align*}
&I_{\mathbb{I}}=\frac{I+\sigma J+\tau K}{2};\quad
&I_{\mathbb{O}}=\frac{J+K}{\sqrt 2};\quad
&\omega=\frac{-1+I+J+K}{2};\\
&I_{\mathbb{T}}=I;\quad
&e_n=e^{\frac{\pi I}{n}}.
\end{align*}
and define the finite subgroups of the sphere $\mathbb S^3$  generated as follows:
\begin{align*}
& 2\mathbb{I}=\langle I_{\mathbb{I}}, \omega\rangle,\ \ \ 
2\mathbb{O}=\langle I_{\mathbb{O}}, \omega\rangle,\ \ \ 
2\mathbb{T}=\langle I_{\mathbb{T}}, \omega\rangle,\\
&2D_{2n}= \langle e_n, J\rangle,\ \ \ 
2C_n=\langle e_n\rangle,\ \ \  
1C_n=\langle e_{\frac{n}{2}}\rangle\ \textnormal{($n$ odd)}.
\end{align*}
The following result holds (see, e.g., \cite[Theorem 12, page 33]{conway}).

\begin{theorem}\label{subgroups}
Every finite subgroup of the sphere $\mathbb S^3$ of unitary quaternions is conjugated to a subgroup of the following list:
$$
2\mathbb{I},\ \ \  2\mathbb{O},\ \ \  2\mathbb{T}, \ \ \  2D_{2n}, \ \ \  2C_n, \ \ \  1C_n \textnormal{($n$ odd)}.
$$
\end{theorem}
\noindent Recall the following
\begin{definition} \label{unipotent} A group $G$ is said to be \emph{unipotent} if all of its elements are unipotent, i.e.,  for all  $g\in G$,  there exists $n\in \mathbb N$ such that $(g-1)^n=0$. 
\end{definition}
\noindent We are now ready to state and prove 

\begin{theorem}\label{finale} If the subgroup $\Gamma\subseteq {\it Aff}(2,\H)$
acts freely and properly discontinuously  on $\H^2$, and  $\H^2/\Gamma$ is compact,  then $\Gamma$ contains a unipotent normal subgroup $\Gamma_0$ of finite index such that $\Gamma/\Gamma_0$ is isomorphic to  a finite subgroup of $\mathbb{S}^3$. 
\end{theorem}
\begin{proof}
We can assume from the previous results that, up to conjugation, $\Gamma$ is contained in $G_2$. Suppose first that $\Gamma$ contains a central element $A$ of the form 
$ \begin{pmatrix}
1& b&r\\
0&d&s\\
0&0&1\\
\end{pmatrix}$ with $d\neq 1$. Conjugate $\Gamma$ by
$$ M=\begin{pmatrix}
1& b(d-1)^{-1}&0\\
0&1&(1-d)^{-1}s\\
0&0&1\\
\end{pmatrix}\in G_2$$
 then $M^{-1}AM= \begin{pmatrix}
1& 0&r'\\
0&d&0\\
0&0&1\\
\end{pmatrix}$.
Since this element is central in $M^{-1}\Gamma M$, all its elements of are of the form $\begin{pmatrix}
1& 0&u\\
0&h&0\\
0&0&1\\
\end{pmatrix}$. And this is in contradiction with Lemma \ref{base} since the translational parts of $\widetilde{\Gamma}$ form a basis for $\H^2$. Thus, for any central element in $\Gamma$, $d=1$.
Since $\Gamma$ is the fundamental group of the compact manifold $\H^2/\Gamma$, it is finitely generated. Let 
\[A_i= \begin{pmatrix}
1& b_i&r_i\\
0&d_i&s_i\\
0&0&1\\
\end{pmatrix}, \quad \text{for $1\leq i\leq k,$}\]  be the set of generators of $\Gamma$. 
If $A_i$ is central, $d_i=1$. If $A_i$ is not central by Lemma \ref{nonabel} we have that $d_i$ is a  root of unity. 
Consider then the homomorphism $\varphi:\Gamma\rightarrow \mathbb{S}^3$
defined as 
\[\varphi(A_i)=\varphi \begin{pmatrix}
1& b_i&r_i\\
0&d_i&s_i\\
0&0&1\\
\end{pmatrix}=d_i.\]
Let $\Gamma_0$ denote the kernel of $\varphi$. Then $\Gamma_0$ is normal and unipotent, and $\Gamma/\Gamma_0$ is isomorphic to a subgroup of $\mathbb{S}^3$. 
Now recall that if $\widetilde{\Gamma}$ denotes the subgroup of pure translation in $\Gamma$, by Theorem \ref{auslander}, $\Gamma/\widetilde{\Gamma}$ is finite. Taking into account that $\widetilde{\Gamma}\subseteq \Gamma_0$, we conclude that $\Gamma/\Gamma_0 \subseteq \Gamma/\widetilde{\Gamma}$ is isomorphic to a finite subgroup of $\mathbb{S}^3$.
\end{proof}
Let us give an explicit example of affine quaternionic manifold of quaternionic dimension $2$.
\begin{example}
 Consider the following transformations in ${\it Aff}(2,\mathbb{H})$:
\begin{align*}
& A = \left(\begin{matrix} 1 & 0 & 0\\0 & 1 & 1 \\0 & 0 & 1
\end{matrix}\right),\quad
  B =\left(\begin{matrix} 1 & -1 & 0\\0 & 1 & I \\0 & 0 & 1
\end{matrix}\right),\quad
 C =\left(\begin{matrix} 1 & 0 & 1\\0 & 1 & 0 \\0& 0 & 1
\end{matrix}\right), \\
& D=\left(\begin{matrix} 1 & 0 & J\\0 & 1 & 0 \\0 & 0 & 1
\end{matrix}\right),\quad
S=\left(\begin{matrix} 1 & 0 & \frac J 2\\0 & -1 & I \\ 0 & 0 & 1
\end{matrix}\right),
\end{align*}
where $I,J,K,L\in\mathbb{S}$ are imaginary units.
If $\Gamma_0=\langle A,B,C,D\rangle$ and $\Gamma=\langle A,B,C,D,S\rangle$, then $\H^2/\Gamma_0$ and $\H^2/\Gamma$ are affine quaternionic manifolds. In particular, they
are the quaternionic analogs of examples (f) and (b) in Vitter's paper.
\vskip .3cm

\noindent To identify all possible subgroups $\Gamma_0$ appearing in Theorem \ref{finale}, we need to recall here 

\begin{definition} A subgroup $H$ of a Lie group $G$ is called \emph{uniform} if $G/\overline H$ is compact ($\overline{H}$ is the closure of $H$).
\end{definition}
\noindent The following results, used in the sequel, concern nilpotent Lie groups and nilmanifolds.
 
\begin{theorem}\cite[Theorem 7.1]{Goldman}.\label{Goldman}
Let $M$ be a compact complete affine manifold with nilpotent fundamental group. Then $M$ is an affine nilmanifold. 
\end{theorem}

\begin{theorem}\cite[Corollary, p. 293]{Mal}.\label{Malcoro}
Nilmanifolds having isomorphic fundamental groups are isomorphic.
\end{theorem}

\begin{theorem} \cite[Theorem 6, p. 296]{Mal}.
\label{malcev} A group $\Lambda$ is isomorphic to a uniform, discrete subgroup  in a connected, simply connected, nilpotent Lie group if, and only if,
\begin{enumerate}
\item $\Lambda $ is finitely generated;
\item $\Lambda$ is nilpotent;
\item $\Lambda$ has no torsion.
\end{enumerate} 
\end{theorem}
\noindent Now, if the subgroup $\Gamma_0\subseteq \Gamma\subseteq {\it Aff}(2,\H)$ is, as in Theorem \ref{finale}, finitely generated, without torsion elements and unipotent - and thus nilpotent - then it is isomorphic via a map $i$ to a uniform, discrete (closed) subgroup $i(\Gamma_0)$ of a unique connected, simply connected, nilpotent Lie group $N$.  
Now the fact that  $\Gamma/\Gamma_0$ is finite and $\H^2/\Gamma_0$ is compact forces $N$ to have real dimension $8$.  Since $\Gamma_0$ is nilpotent, then Theorem \ref{Goldman} implies that $\H^2/\Gamma_0$  is an affine nilmanifold. On the other hand, the nilmanifolds  $\H^2/\Gamma_0$ and $N/i(\Gamma_0)$ having isomorphic fundamental groups, are isomorphic thanks to Theorem \ref{Malcoro}. As a consequence, the quotient $N/i(\Gamma_0)$ is an affine quaternionic manifold; thus it is hypercomplex, and the same holds for its covering $N$ (see Remark \ref{hypercomplex}). 
Summarizing, 

\begin{corollary} In the hypotheses of Theorem \ref{finale}, the unipotent subgroup $\Gamma_0$ of $\Gamma$ turns out to be isomorphic to a uniform, discrete subgroup $i(\Gamma_0)$ of a suitable nilpotent, hypercomplex, connected, simply connected, $8$-dimensional Lie group $N$. Moreover $N/i(\Gamma_0)$ is affine quaternionic, and the same holds for $N$.
\end{corollary}

\noindent Dotti and Fino, in  \cite{DF}, give a classification of all possible hypercomplex, simply connected, nilpotent lie groups of real dimension 8.  In this classification one can find, in particular, those $N$ that carry an affine quaternionic structure. Indeed, all hypercomplex, simply connected, nilpotent lie groups of real dimension 8 carry such a structure, as stated in 
\begin{theorem}\cite[Proposition 2.1]{DF2}.
Every hypercomplex structure on a $8$-dimensional, nilpotent Lie group is flat and hence affine quaternionic.
\end{theorem}
At this point, the next step would be to find explicitly, and possibly list, all subgroups $\Gamma\subseteq {\it Aff}(2,\H)$ acting freely and properly discontinuously  on $\H^2$, and  such that $\H^2/\Gamma$ is compact. In order to do this, one could exploit the mentioned classification of Dotti and Fino: first identify all uniform discrete subgroups  of a nilpotent, hypercomplex (or equivalently affine quaternionic), simply connected, $8$-dimensional Lie Group $N$, and then restrict to those such that $\Gamma/\Gamma_0$ is one of the finite subgroups of $\mathbb S^3$ listed above. In the next section we apply this procedure to the case of Abelian hypercomplex structures.

\section{The Abelian case}
Let $N$ be a nilpotent, hypercomplex, simply connected, $8$-dimensional Lie group, let $\mathfrak{n}$ denote its Lie algebra, and recall that a complex structure $J$ on $N$ is called {\em Abelian} if $J$ satisfies the condition $[X,Y] = [JX,JY]$ for all $X,Y\in\mathfrak{n}$. We will start by considering here the case in which the group $N$ admits an Abelian hypercomplex structure, i.e., three abelian complex structures.


In order to state the mentioned results by Dotti and Fino, \cite{DF1}, we  first recall the definition of groups of Heisenberg type (for short, type $H$)
classified in \cite{Kaplan}.
The  $(2n + 1)$-dimensional real Heisenberg group admits two realisations: one, as a subgroup of $GL(n + 2, \R)$
\begin{equation}\label{reale}
H(n,\R)=\left\{\left(\begin{matrix} 1 & a & c\\0 & I & b^T \\0 & 0 & 1
\end{matrix}\right):a,b\in \R^n,c\in \R \right\}
\end{equation}
and the other as a subgroup of $GL(n + 2, \C)$
\begin{equation}\label{complessa}
H_1(n)=\left\{\left(\begin{matrix} 1 & z & Im w\\0 & I & -\bar{z}^T \\0 & 0 & 1
\end{matrix}\right):z\in \C^n,w\in \C \right\}
\end{equation}
Using complex coefficients in \eqref{reale} and quaternionic coefficients in \eqref{complessa} one obtains the complex Heisenberg groups of matrices
$$
H_2(n)=H(n,\C)=\left\{\left(\begin{matrix} 1 & \xi & z\\0 & I & \omega^T \\0 & 0 & 1
\end{matrix}\right):\xi,\omega\in \C^n,z\in \C \right\}
$$
and the quaternionic Heisenberg groups of matrices

$$
H_3(n)=\left\{\left(\begin{matrix} 1 & q & Im h\\0 & I & -\bar{q}^T \\0 & 0 & 1
\end{matrix}\right):q\in \H^n,h\in \H \right\}
$$
The three families $H_1(n), H_2(n)$ and $H_3(n)$ so obtained are two step nilpotent groups with centers of dimension $1,2$ and $3$ respectively. The groups $H_1(n)$ and $H_3(n)$ correspond to the nilpotent part in the Iwasawa decomposition of the isometry group of the complex hyperbolic space and quaternionic hyperbolic space respectively.
The groups $H_1(2), H_2(1)$ and $H_3(1)$ are the real, complex and quaternionic Heisenberg groups of dimension $5, 6, 7$,  respectively.

\noindent In \cite{DF1} Dotti and Fino prove that a nilpotent $8$-dimensional simply connected Lie group admitting an invariant Abelian hypercomplex structure is either a real Euclidean space or a trivial extension of a group of type $H$; indeed using the following 
\begin{lemma}Let N be an $8$-dimensional nilpotent Lie group with an Abelian hypercomplex structure. Then $N$ is either Abelian, or $2$-step nilpotent with a $4$-dimensional center.
\end{lemma}
\noindent they thus obtain
\begin{theorem}\cite[Theorem 4.1]{DF1} Let $N$ be an $8$-dimensional nilpotent Lie group. Then $N$ carries an Abelian hypercomplex structure if and only if $N$ is either Abelian or isomorphic to a  trivial extensions of a group of type $H$ with center of dimension $i = 1, 2$ or $3$, i.e., respectively, to $N_1 = \R^3 \times H_1(2)$,
$N_2 = \R^2 \times H_2(1)$ or $N_3 = \R \times H_3 (1)$.
\end{theorem}

\noindent The groups $N_1$, $N_2$ can carry only Abelian hypercomplex structures since their commutator subgroups  are $1$-dimensional and $2$-dimensional respectively. On the other hand $N_3$ does admit also non-Abelian hypercomplex structures. 

We will now turn our attention to find all uniform, discrete subgroups $L$ of $N_1$.  
In this case, the uniform discrete subgroups can be completely classified using  a result in \cite{GW}. 

Indeed,
for $x,y$ row vectors in $\mathbb{R}^n,$ let $$\lambda(x,y,t)=\left\{\left(\begin{matrix} 1 & x& t\\0 & I & y^T\\0 & 0 & 1
\end{matrix}\right) \right\}
$$
\begin{definition} \label{discreti} For $r=(r_1,r_2,\ldots,r_n)\in {(\mathbb{Z}^+)}^n$ such that $r_j$ divides $r_{j+1}$ with $1\leq j\leq n$, let $r{\mathbb{Z}}^n$ denote the $n$-uples $x=(x_1,x_2,\ldots,x_n)$ for which $x_i\in r_i\mathbb{Z}$ with  $1\leq i\leq n$. Define $$\Lambda_r=\{\lambda(x,y,t):\;\;x\in r{\mathbb{Z}}^n,y\in{\mathbb{Z}}^n,t\in\mathbb{Z} \}$$
\end{definition} \noindent It follows easily that  $\Lambda_r$ is  a uniform discrete subgroup of $H_1(n).$
With this notation it is proved that
\begin{theorem} \cite[Theorem 2.4]{GW}. The subgroups $\Lambda_r$ classify the uniform discrete subgroups of $H_1(n)$ up to automorphism. That is, if $\Lambda$ is any  uniform discrete subgroup of $H_1(n)$ then there exists a unique $r$ for which some automorphism of $H_1(n)$ maps $\Lambda$ to $\Lambda_r$. Moreover for all $r,s$ as in Definition \ref{discreti}, $\Lambda_r$ and $\Lambda_s$ are isomorphic if and only if $r=s$.
\end{theorem}
\noindent One can easily see that the  generic discrete uniform  subgroup of $H_1(2)$ corresponds to  those 
appearing in Fillmore Sheuneman's  paper, \cite[Theorem 4.1]{FS}.\\ 
\noindent Let now $\Gamma_0 \subset \Gamma$ subgroups of ${\it Aff}(2,\H)$ be as in Section 3. For the sake of simplicity, we will use the same names to indicate their isomorphic images in $N$.
\noindent A discrete uniform subgroup $\Gamma_0$ of $N_1$ must be isomorphic to one of the form $\Theta\times\Lambda_r$,  where $\Theta$ is a lattice in $\R^3;$ 
thus the corresponding group $\Gamma$ is the finite extension of $\Gamma_0$ via a  finite subgroup (see the list) of $\mathbb{S}^3.$  And this concludes the case of $N_1$.
\vskip .3cm

For the groups of type  $H_3(n)$ a class of discrete uniform subgroups is given by the lattices
 $\Lambda(1, n)\subseteq H_3(n)$ consisting
of matrices with integer entries; we denote an element of $\Lambda(1, n)$ by
$(a_1, b_1, . . . , a_n, b_n, c).$ For $m\in \mathbb{N}$, and  $m\geq  2$, let us take the sublattice
$\Delta(1,n; m) = \{(a_1, b_1, . . . , a_n, b_n, c)\in \Lambda(1, n)\; | \; a_1 \in m\Z\}$ is contained  in $\Lambda(1, n)$;
then $H_3(n)/\Delta(1, n; m)$ is again a nilmanifold and there is an $m : 1$ covering $H_3(n)/\Delta(1, n; m)\rightarrow N(1, n)=
H_3(n)/\Lambda(1,n)$ with deck transformation group
$\Lambda(1,n)/\Delta(1, n; m)= \Z_m$.

As an application, we are able to exhibit other classes of examples of affine quaternionic surfaces when of $N=N_3= \R \times H_3 (1)$, where, explicitly,
$$
H_3(1)=\{\left(\begin{matrix} 1 & a+ib+jc+kd & i\alpha+j\beta+k\gamma\\0 & 1 & -a+ib+jc+kd\\0 & 0 & 1
\end{matrix}\right):a,b,c,d,\alpha,\beta,\gamma\in \R\}
$$
When we take  $\Lambda(1,3)$ as a discrete, uniform subgroup of $H_3(1)$, 
$$
\Lambda(1,3)=\{\left(\begin{matrix} 1 & a_1+ib_1+ja_2+kb_2 & ia_3+jb_3+kc\\0 & 1 & -a_1+ib_1+ja_2+kb_2\\0 & 0 & 1
\end{matrix}\right):a_i,b_i,c\in \Z\}
$$
we obtain that $\Gamma_0=\Z \times \Lambda(1,3)$ and, again, $\Gamma$ will be a finite extension of $\Gamma_0$ via a  finite subgroup (see the list) of $\mathbb{S}^3.$ A second class of examples can be constructed by taking  $\Gamma_0=(\Z\times \Delta(1, 3; m))\subset (\Z \times \Lambda(1,3))$ (with $m\in\mathbb{N}$).
	\end{example}

\end{document}